\documentclass[a4paper, oneside, 11pt]{amsart} 

\usepackage[english]{babel}
\usepackage[utf8]{inputenc}
\usepackage[T1]{fontenc}

\usepackage{amsmath}
\usepackage{amssymb}
\usepackage{bbm}
\usepackage{enumerate}

\usepackage{geometry}
\usepackage{verbatim}
\usepackage{graphicx}
\usepackage{color}
\usepackage{hyperref}

\newtheorem{thm}{Theorem}[section]
\newtheorem{cor}[thm]{Corollary}
\newtheorem{lem}[thm]{Lemma}

\newtheorem*{q}{Question}
\newtheorem{rem}[thm]{Remark}
\newtheorem{ex}[thm]{Example}

\newcommand\C{\mathbb C} 

\newcommand\R{\mathbb R} 
\newcommand\Z{\mathbb Z} 

\renewcommand{\d}{\mathrm d}
\newcommand{\del}[1]{\partial #1}

\renewcommand{\leq}{\leqslant}
\renewcommand{\geq}{\geqslant}

\newcommand{\D}{\operatorname{D}}
\newcommand{\T}{\operatorname{T}} 
\renewcommand{\L}{\operatorname{L}} 
\renewcommand{\S}{\operatorname{S}} 
\renewcommand{\H}{\operatorname{H}} 
\DeclareMathOperator{\interior}{int}

\DeclareMathOperator{\Wh}{Wh}
\DeclareMathOperator{\GL}{GL}
\renewcommand{\O}{\operatorname{O}}

\title{Contact manifolds and Weinstein h-cobordisms}

\author{Sylvain Courte}
\address{UMPA, ENS de Lyon, 46 all\'ee d'Italie 69364 Lyon Cedex 07, FRANCE}
\email{sylvain.courte@ens-lyon.fr}
\urladdr{http://perso.ens-lyon.fr/sylvain.courte/}

\begin{document}

\begin{abstract}
We prove that closed connected contact manifolds of dimension $\geq 5$ related by a flexible Weinstein h-cobordism become contactomorphic after some kind of stabilization. We also provide examples of non-conjugate contact structures on a closed manifold with exact symplectomorphic symplectizations.
\end{abstract}

\maketitle

\tableofcontents

\section{Introduction} 

This paper is a sequel to \cite{C2014}, in which the following phenomenon was observed. If two closed contact manifolds of dimension~$\geq~5$ are related by a flexible Weinstein h-cobordism, then their symplectizations are
exact symplectomorphic. As observed in \cite{C2014} such contact manifolds need not even be diffeomorphic, but we may ask~:

\begin{q}
If two contact structures on a given closed manifold have exact symplectomorphic symplectizations, are they conjugate by a diffeomorphism ?
\end{q}

In this paper we wish to provide partial answers to this question in two different directions. On one hand we prove that contact manifolds related by a flexible Weinstein h-cobordism become contactomorphic after some kind of \emph{stabilization}. Our inspiration comes from the following fact noticed by Hatcher and Lawson in \cite{MR0415640}. Let $M$ and $M'$ be h-cobordant closed connected manifolds of dimension $m$ and let $k$ be any integer satisfying $2 \leq k \leq m-2$, then for $l$ large enough $M \# (\S^k\times \S^{m-k})^{\# l}$ is diffeomorphic to $M' \# (\S^k \times \S^{m-k})^{\# l}$ (where $^{\# l}$ denotes the connected sum iterated $l$ times). We will prove in section \ref{sec:stable} a contact analogue of this result using Morse-Smale theory of Weinstein structures developped by Cieliebak and Eliashberg (see \cite{CE2012}). On the other hand we prove that the answer to the question, as stated, is negative due to the following phenomenon : there are contact structures on a given manifold which are not conjugate as almost-contact structures but have exact symplectomorphic symplectizations, this is the content of section \ref{sec:almost}.

\textbf{Acknowledgments:}
I warmly thank Emmanuel Giroux for his support and François Laudenbach for his interest in this work and his comments on a previous draft of this paper.

\section{A stabilization theorem}\label{sec:stable}

\subsection{Hatcher's and Lawson's remark}

Let us briefly explain the remark by Hatcher and Lawson mentioned in the introduction. Let $(W, M, M')$ be an h-cobordism of dimension $m+1 \geq 6$. For all $2 \leq k \leq m-2$, there is an ordered Morse function on $W$ with only critical points of index $k$ and $k+1$. Let $N$ be a level set separating the critical points of index $k$ and $k+1$. Since the homology of the pair $(W, M)$ vanishes there must be an equal number $l$ of critical points of each index. The key point is that, in such a situation, handles of index $k$ are trivially attached to $M$ and, dually, handles of index $k+1$ are trivially attached to $M'$; by that we mean that the attaching spheres bound disks and have trivial normal framings (induced by the disks). In particular the level set $N$ is diffeomorphic to $M \# (\S^k \times \S^{m-k})^{\# l}$ as well as to $M' \# (\S^k \times \S^{m-k})^{\# l}$. In \cite{MR0415640}, this key point is proved using \emph{Smale's trading trick} which consists in replacing a critical point of index $k$ by a critical point of index $k+2$ (birth of a pair of critical points of index $(k+1,k+2)$ followed by the death of a pair of critical points of index $(k, k+1)$); the fact that the critical points of index $k$ can be cancelled with a critical point of index $k+1$ implies that its attaching sphere is trivial (see lemma \ref{cancel} for a proof in a contact setting). In fact, in the extreme case $k=2$, it is not proved that the $3$-handle is trivially attached to $M'$ because the critical points of index $3$ cannot be replaced by a critical point of index $1$ (likewise in the case $k=m-2$). In the context of Weinstein structures of dimension $2n$, the trading trick cannot work for a critical point of index $n-1$ because it would have to be replaced by a critical point of index $n+1$, so we will use a different argument which has the advantage to treat the extreme cases $k=2$ and $k=m-2$ as well.

\subsection{Main results and proofs}
For $n\geq 3$ and $ 2 \leq k \leq n-1$ we consider the (subcritical) Liouville manifold:
\[(\T^* \S^k \times \R^{2(n-k)}, \lambda = p \d q +\frac{1}{2} \sum_{i=1}^{n-k} r_i^2 \d \theta_i).\]
where $p \d q$ is the canonical $1$-form on $\T^* \S^k$ and $(r_i, \theta_i)$ are multipolar coordinates in $\R^{2(n-k)}$. The contact manifold at infinity of this Liouville manifold is diffeomorphic to $\S^k \times \S^{2n-k-1}$; we will always consider this contact structure on $\S^k \times \S^{2n-k-1}$. Note that, as it follows from Weinstein tubular neighborhood theorem, this contact manifold is the model for the boundary of a small tube around any isotropic sphere $\S^k$ with trivial symplectic normal bundle in a symplectic manifold of dimension $2n$.

\begin{thm}\label{main:stable}
Let $(M, \xi)$ and $(M', \xi')$ be closed connected contact manifolds of dimension $2n-1 \geq 5$. Assume there is a flexible Weinstein h-cobordism $W$ from $(M, \xi)$ to $(M', \xi')$. Denote by $l$ the minimal integer such that the Whitehead torsion of $W$ is represented by a matrix of size $l$. Then for any integer $k$ satisfying $2\leq k \leq n-1$, we have:
\[M\#(\S^k \times \S^{2n-k-1})^{\#l} \text{ is contactomorphic to } M' \# (\S^k \times \S^{2n-k-1})^{\#l}\]
\end{thm}

In the statement above, the symbol $\#$ denotes the contact connected sum.

For contact manifolds that are already "sufficiently stabilized", we get the following partial answer to the question raised in the introduction.

\begin{cor}\label{main:cor}
Let $(M, \xi)$ be a closed connected contact manifold of dimension $2n-1 \geq 5$ contactomorphic to $(N, \zeta) \# (\S^k \times \S^{2n-k-1})^{\#l}$ for some closed contact manifold $(N, \zeta)$ and some integers $l\geq 0$ and $2 \leq k \leq n-1$. Assume that the map $\GL_l(\Z[\pi_1 M]) \to \Wh(\pi_1 M)$ is surjective. Then any contact manifold $(M', \xi')$ related to $(M, \xi)$ by a flexible Weinstein h-cobordism is contactomorphic to it.
\end{cor}

\begin{rem}
\begin{enumerate}
\item For $n \geq 4$ and $k \leq n-2$, we can consider only subcritical Weinstein structures instead of the broader class of flexible ones.
\item As follows from the proof of the $s$-cobordism theorem, the minimal integer $l$ in the statement above equals the minimal number of critical points of index $k$ for a Morse function in normal form of index $(k,k+1)$ (for any $2 \leq k \leq 2n-2$) and also half the minimal number of critical points of any Morse function.
\item For finite cyclic fundamental groups $\pi$, the map $\GL_1(\Z[\pi]) \to \Wh(\pi)$ is surjective, so one connect sum with $\S^k \times \S^{2n-k-1}$ is enough.
\item If two closed contact manifolds have exact symplectomorphic symplectizations, then they are related by an invertible Liouville cobordism. However we do not know whether these invertible Liouville cobordisms are necessarily Weinstein flexible so that theorem \ref{main:stable} applies.
\end{enumerate}
\end{rem}

\begin{ex}
The manifolds $M_1=\L(7,1) \times \S^2$ and $M_2=\L(7,2) \times \S^2$ carry contact structures with exact symplectomorphic symplectizations though they are not diffeomorphic (see \cite{C2014}). It follows from theorem \ref{main:stable} that $M_1 \# \S^2 \times \S^3$ is contactomorphic to $M_2 \# \S^2 \times \S^3$ where $\S^2 \times \S^3 \simeq \del_\infty( \T^*\S^2 \times \R^2)$. From corollary \ref{main:cor}, we also get that for each flexible Weinstein h-cobordism $(W, M_1\# \S^2 \times \S^3, M')$, $M'$ is contactomorphic to $M_1 \# \S^2 \times \S^3$.
\end{ex}

The main tools for the proof of theorem \ref{main:stable} and corollary \ref{main:cor} are the flexibility results of Cieliebak and Eliashberg concerning Weinstein structures. For the sake of brevity, we will often refer directly to the book \cite{CE2012} instead of repeating here many statements.

We start with a lemma.
\begin{lem}\label{cancel}
Let $(W, \omega, X, \phi)$ be a connected Weinstein cobordism of dimension $2n$ from $M$ to $M'$ such that $\phi$ has only two critical points $p$ and $q$ of index $k+1$ and $k$ respectively, with $\phi(q)<\phi(p)$ and such that, in an intermediate level set $N$ between $p$ and $q$, the ascending sphere of $q$ intersects the descending sphere of $p$ transversally in a single point. Then $N$ is contactomorphic to $M\# \S^k\times \S^{2n-k-1}$ as well as to $M' \# \S^k\times \S^{2n-k-1}$.
\end{lem}

\begin{proof}
\textit{Step 0 : Cancellation.}

According to  proposition 12.22 in \cite{CE2012}, there is a Weinstein homotopy from $(\omega,X,\phi)$ to a Weinstein structure without critical points. In particular $M$ and $M'$ are contactomorphic (and connected) and we only need to prove that $N$ is contactomorphic to $M \# \S^k \times \S^{2n-k-1}$.

\textit{Step 1:}
By a Weinstein homotopy we create a pair of critical points $r$ and $s$ of index $1$ and $0$ respectively below $q$ (see proposition 12.21 in \cite{CE2012}). The intersection of the ascending disc of $s$ with a level set $P$ between $q$ and $r$ is an open disc $D$ of codimension zero in $P$.

\textit{Step 2:}
After a Weinstein homotopy, we can assume that $X$ is standard near $p$ and $q$ (see proposition 12.12 in \cite{CE2012}). The closure of the descending disc of $p$ then intersects $P$ in an isotropic closed disk $D'$ of dimension $k$. Since $P$ is connected, there is a contact isotopy of $P$ which takes $D'$ inside $D$. We realize this contact isotopy by a Weinstein homotopy which is fixed up to scaling above $P$ using lemma 12.5 from \cite{CE2012}.

\begin{figure}
\centering
\def\svgwidth{400pt}
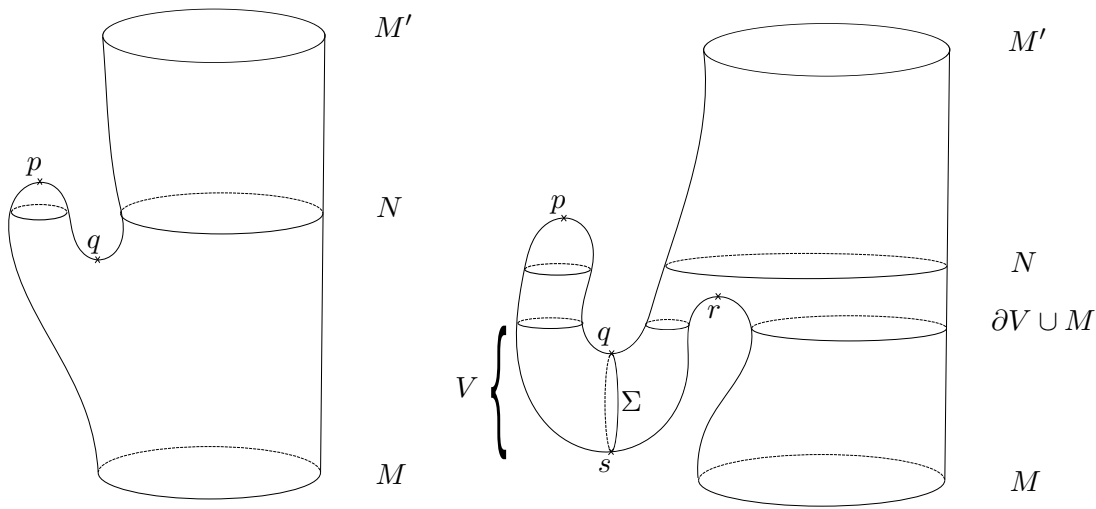
\caption{Picture of $W$ before and after the Weinstein homotopy.}
\label{fig-cancel}
\end{figure}

\textit{Step 3:}
By a Weinstein homotopy we lower $q$ to a level set between $f(r)$ and $f(s)$.  Denote by $V$ the connected component containing $s$ of a sublevel set just below $r$ (see figure \ref{fig-cancel}). We obtain a Weinstein cobordism from $M \cup \del V$ to $N$ with only one critical point $r$ of index $1$ and whose descending disc intersects both $M$ and $\del V$, $N$ is therefore contactomorphic to $M \# \del V$.

\textit{Step 4:}
We now prove that the boundary of $V$ is contactomorphic to $\S^k\times \S^{2n-k-1}$. After a Weinstein homotopy supported in a neighborhood of $s$, we can assume (see proposition 12.12 in \cite{CE2012}) that the Weinstein structure is conform to the model:
\[(\omega = \sum_{i=1}^n\d x_i \wedge \d y_i, X =\frac{1}{2} \sum_{i=1}^n x_i \del_{x_i} + y_i \del_{y_i}, \phi  = \phi(s)+ \sum_{i=1}^n x_i^2 + y_i^2).\]

The closure of the descending disc of $p$ intersects a small sphere around $s$ in an isotropic closed disc $D''$. There is a contact isotopy of this sphere which takes $D''$ to the disc given by:
\[\{x_{k+2}= \dots = x_n =0, x_{k+1} \geq 0, y_1 = \cdots = y_n = 0 \},\]
that we realize by a further Weinstein homotopy using lemma 12.5 from \cite{CE2012}. Now the closure of the descending disc of $p$ is an embedded disc of dimension $k+1$ whose boundary is the skeleton $\Sigma$ of $V$. In particular $\Sigma$ is an embedded isotropic sphere with trivial symplectic normal bundle. We claim that there is a small riemannian tube around $\Sigma$ whose boundary is transverse to $X$ and conclude that $\del V$ is contactomorphic to $\S^k \times \S^{2n-k-1}$.
\end{proof}

We now prove theorem \ref{main:stable}.

\begin{proof}[Proof of theorem \ref{main:stable}]

\textit{Step 1 : Reducing to a normal form.}

According to the proof of the s-cobordism theorem (see \cite{K65}), there is a path $\phi_s$ of functions with birth-death type accidents and critical points of index less or equal to $n$ such that $\phi_0=\phi$ and $\phi_1$ has a regular level set $N$ with $l$ critical points $p_1, \dots, p_l$ of index $k+1$ above $N$, $l$ critical points $q_1,\dots, q_l$ of index $k$ below $N$ and no other critical points. According to theorem 14.1 in \cite{CE2012} there is a Weinstein homotopy $(\omega_s, X_s, \phi_s)_{s \in [0,1]}$ of flexible Weinstein structures starting from $(\omega, X, \phi)$. After a perturbation of $X_1$ we can also assume that is is Morse-Smale; we rename $(\omega_1,X_1,\phi_1)$ back to $(\omega, X, \phi)$.

\begin{figure}[!h]
\centering
\def\svgwidth{300pt}
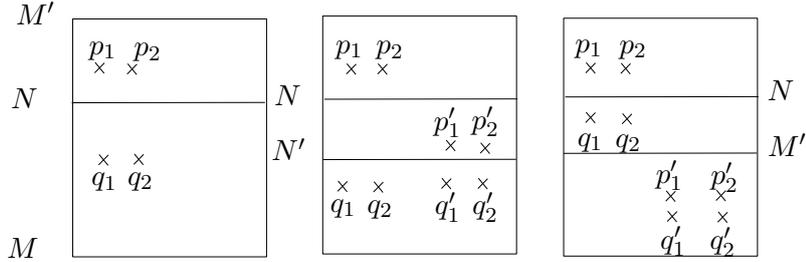
\caption{Schematic picture of the successive Weinstein structures on the cobordism $W$ with $l=2$.}
\label{slides}
\end{figure}

\textit{Step 2 : Creating cancelling pairs of critical points.}

By a Weinstein homotopy of flexible Weinstein structures provided (see proposition 12.21 in \cite{CE2012}) we create $l$ cancelling pairs of critical points of index $k$ and $k+1$ below $N$, denoted respectively $q'_1,\dots,q'_l$ and $p'_1, \dots, p'_l$, with no trajectories joining the critical points $q_1, \dots, q_l$ and $p_1, \dots, p_l$ (see figure \ref{slides}). The effect on the Morse complex is as follows. In a universal cover $\tilde{W} \to W$ with automorphism group $\pi\simeq \pi_1 W$, the Morse complex of $(X, \phi)$ is a chain complex over the ring $\Z[\pi]$ which looks like~:
\[0 \to C_{k+1} \overset{\partial_{k+1}}{\longrightarrow} C_k \to 0.\]
By choosing lifts $\tilde p_i$, $\tilde q_i$, $\tilde p'_i$ and $\tilde q'_i$ of the critical points of $\phi$ to $\tilde{W}$ and orientations for unstable manifolds at each critical point, we obtain bases $(\tilde p_1,\dots,\tilde p_l, \tilde p'_1, \dots, \tilde p'_l)$ of $C_{k+1}$ and $(\tilde q_1, \dots, \tilde q_l, \tilde q'_1, \dots, \tilde q'_l)$ of $C_k$. The corresponding matrix of $\partial_{k+1}$ is the stabilized matrix

\[\begin{pmatrix} A & 0 \\ 0 & 1 \end{pmatrix} \in \GL_{2l}(\Z[\pi]).\]

with $A \in \GL_l(\Z[\pi])$.

\textit{Step 3 : A few handleslides.}

Take an intermediate level set $N'$ separating index $k$ and index $k+1$ critical points. In the cobordism between $M$ and $N'$, there are only critical points of index $k$. We claim that there is a homotopy of gradient-like vector field $Y_t$ for $\phi$ such that $Y_0=X$, $Y_t=X$ above $N'$ and such that the boundary operator $\del_{k+1}$ for $Y_1$ has matrix 

\[\begin{pmatrix}
1 & 0 \\
0 & A
\end{pmatrix}\in \GL_{2l}(\Z[\pi]).\]

Indeed, one can realize this homotopy by a sequence of handleslides (see \cite{K65}) between critical points of index $k$ corresponding to the following \emph{row} operation on matrices~:

\[\begin{pmatrix}
A & 0 \\ 0 & 1
\end{pmatrix} \to
\begin{pmatrix}
A & -1 \\ 0 & 1
\end{pmatrix} \to
\begin{pmatrix}
A & -1 \\ A & 0
\end{pmatrix} \to
\begin{pmatrix}
0 & -1 \\ A & 0
\end{pmatrix} \to
\begin{pmatrix}
0 & -1 \\ A & A
\end{pmatrix} \to
\begin{pmatrix}
1 & 0 \\ A & A
\end{pmatrix} \to
\begin{pmatrix}
1 & 0 \\ 0 & A
\end{pmatrix},\]

According to lemma 14.10 in \cite{CE2012}, there is a flexible Weinstein homotopy $(\omega_s, X_s, \phi)$ which is fixed up to scaling above $N'$ and such that $X_1$ is homotopic to $Y_1$ in the space of Morse-Smale gradient-like vector fields for $\phi$. In particular, the boundary operator $\del_{k+1}$ for $X_1$ and $Y_1$ are equal. Rename $(\omega_1,X_1, \phi)$ back to $(\omega, X, \phi)$.

\textit{Step 4 : Applying the Whitney trick}

Since the $Z[\pi]$ intersections numbers of descending spheres of $q'_1, \dots, q'_l$ with ascending spheres of $p_1, \dots, p_l$ are zero, we can apply the Whitney trick to make them disjoint by a smooth isotopy. By the flexibility hypothesis, the descending spheres are loose (or subcritical) and can therefore be made disjoint by legendrian isotopy using Murphy's h-principle (see \cite{M2012a}) (or Gromov's h-principle see \cite{CE2012} theorem 7.11). We can then raise the critical values of $p_1, \dots, p_l$ above the critical values of $q'_1, \dots, q'_l$. 
Now in the cobordism containing the critical points $p_1, \dots, p_l$ and $q_1, \dots, q_l$, the boundary operator $\del_{k+1}$ in the Morse complex  is the identity matrix. Successive application of the Whitney trick and of lemma~14.11 in \cite{CE2012} allows us to make the critical points $p_1, \dots, p_l$ in cancellation position with $q_1, \dots, q_l$ by a Weinstein homotopy. Inductively applying lemma~\ref{cancel} then shows that $N$ is contactomorphic to $M' \# (\S^k \times \S^{2n-k-1})^{\# l}$.

\textit{Step 5 : repeating everything}

To prove that $N$ is also contactomorphic to $M \# (\S^k \times \S^{2n-k-1})^{\# l}$ we repeat steps 2, 3, 4 analogously \emph{above} $N$. Note that in step $3$ we use analogous \emph{column} instead of row operations on matrices because we do handleslides between critical points of index $k+1$ instead of $k$.
\end{proof}

\begin{proof}[Proof of corollary \ref{main:cor}]
Let $(W, M, M')$ be an h-cobordism with a flexible Weinstein structure inducing $\xi$ and $\xi'$. Denote by $\tau \in \Wh(\pi_1M)$ the Whitehead torsion of $W$. According to the $s$-cobordism theorem, there is an h-cobordism $(V,N,N')$ with Whitehead torsion $\tau$ (we identify $\pi_1 M \simeq \pi_1 N$). Theorem 13.1 in \cite{CE2012} allows us to construct a flexible Weinstein structure on $V$ inducing contact structures $\zeta$ on $N$ and $\zeta'$ on $N'$ (the hypothesis of theorem 13.1 are fulfilled, see \cite{C2014}). According to theorem \ref{main:stable} (the whitehead torsion of $W$ is represented by a matrix of size $l$ because $\GL_l(\Z[\pi]) \to \Wh(\pi)$ is surjective), $(N,\zeta) \# (\S^{k} \times \S^{2n-k-1})^{\# l}$ is contactomorphic to $(N',\zeta') \# (\S^k \times \S^{2n-k-1})^{\# l}$, thus we are led to prove that $(M' ,\xi')$ is contactomorphic to $(N',\zeta')\# (\S^k \times \S^{2n-k-1})^{\# l}$. For this we consider the trivial Weinstein structure on $(\S^k \times \S^{2n-k})^{\# l} \times [0,1]$ and perform a connected sum operation with $V$ along the cobordisms (that is we glue them along a neighbourhood of an arc going from $\del_-$ to $\del_+$). We get a flexible Weinstein cobordism from $(M, \xi)$ to $(N', \zeta')\# (\S^k \times \S^{2n-k-1})^{\# l}$ with Whitehead torsion $\tau$. By the $s$-cobordism theorem, this cobordism is diffeomorphic to $W$ by a diffeomorphism relative to $M$. Since there is only one non-degenerate two form extending $\xi$ up to homotopy (see for example lemma 2.7 in \cite{C2014}), we have two flexible Weinstein structures on $W$ that are formally homotopic and by theorem 14.3 of \cite{CE2012}, we get that $M'$ is contactomorphic to $N'\# (\S^k \times \S^{2n-k-1})^{\# l}$.
\end{proof}

\section{Non-conjugate almost-contact structures}\label{sec:almost}

\begin{thm}\label{main:almost}
For $n \geq 3$, the closed oriented manifold $M^{2n-1} = \L(5,1) \times \S^{2n-4}$ carries two contact structures $\xi$ and $\xi'$ that are not conjugate by a diffeomorphism of $M$ (even as almost-contact structures) but which have exact symplectomorphic symplectizations. Moreover they bound Weinstein structures on $V=\L(5,1) \times \D^{2n-3}$ that are not conjugate as non-degenerate $2$-forms but have exact symplectomorphic completions.
\end{thm}

The topological phenomenon that we will make use of is the following.

\begin{lem}\label{simple}
No diffeomorphism of $M$ may act on $\pi_1 M=\Z/5$ by multiplication by $\pm 2$. The same holds for $V$.
\end{lem}

\begin{proof}
This is an application of simple homotopy theory. We sketch the proof and refer to \cite{Mil61} for more details on Reidemeister torsion. Denote by $\Delta$ the Reidemeister torsion with respect to the ring homomorphism $\Z[\Z/5]=\Z[t]/(t^5-1) \to \C$ that sends $t$ to $\zeta = e^{\frac{i2\pi}{5}}$; this is an element in the quotient group $\C^*/\langle \pm \zeta\rangle$. We have (see \cite{Mil61} p.583, note that the formula for $\Delta$ is the inverse because of a different convention)
\[\Delta(L(5,1))= (\zeta-1)^2,\]

and using the product formula (see \cite{Mil61} p.587), we get
\[\Delta(M)=(\zeta-1)^4, \quad \Delta(V)=(\zeta-1)^2.\]

If $\Psi : M \to M$ is a diffeomorphism inducing multiplication by $\pm 2$ on $\pi_1=\Z/5$, we would have (by invariance of Reidemeister torsion by diffeomorphism)
\[\Psi_* \Delta(M) = (\zeta^{\pm 2}-1)^4 = (\zeta-1)^4=\Delta(M)\]
which is false (these complex numbers have different moduli); and likewise for $V$ in place of $M$.
\end{proof}

\begin{proof}[Proof of theorem \ref{main:almost}]
\textit{Step 1 : Construction of an h-cobordism.}

\begin{figure}[!h]
\centering
\def\svgwidth{250pt}
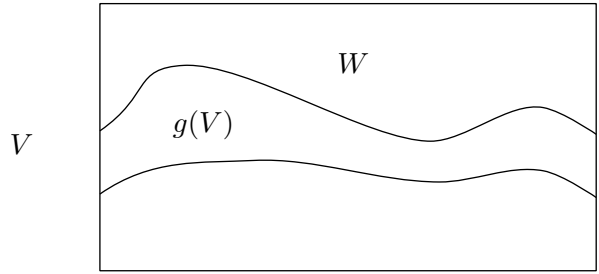
\caption{The h-cobordism $W$.}
\label{fig-hcob}
\end{figure}

The arguments in this step are similar to that in \cite{Mil61}. Note that $M$ is the (oriented) boundary of $V^{2n}=\L(5,1) \times \D^{2n-3}$. According to the homotopy classification of maps between lens spaces (see \cite{dRMK67, Coh73}), there is a homotopy equivalence $f : \L(5,1) \to \L(5,1)$ which induces multiplication by $2$ on $\pi_1$. The map $f\times 0 : \L(5,1) \to V$ is homotopic to an embedding $g$ (by general position for $n \geq 4$ and by Haefliger's embedding theorem \cite{Hae61} for $n =3$). The normal bundle of $g$ is trivial; in fact every real vector bundle of rank $k \geq 3$ on $\L(5,1)$ is trivial because the cohomology groups $\H^i(\L(5,1); \pi_{i-1} \O(k))$ all vanish. Therefore we can extend $g$ to an embedding $V \to \interior V$ (still denoted by $g$); the region $W = V \setminus g(\interior V)$ is a non-trivial h-cobordism from $M$ to $M$ (see figure \ref{fig-hcob}).

\textit{Step 2 : Construction of the Weinstein and contact structures.}

There exists a complex line bundle $\eta \to V$ with $c_1(\eta)\neq 0 \in \H^2(V)\simeq \Z/5$ ($\Z$ coefficients are understood for all homology and cohomology groups appearing in the sequel). The real vector bundle $\eta \oplus \underline{\R}$ is trivial ($\underline{\R}^k$ and $\underline{\C}^k$ denote trivial real and complex vector bundles), as well as the tangent bundle $\T \L(5,1)$ (it follows from the vanishing of the cohomology groups as before). Hence there is a real isomorphism
\[\T V \overset{\sim}{\longrightarrow} \eta \oplus \underline{\C}^{n-1},\]
and we denote by $J$ the pulled back complex structure on $\T V$. We have $c_1(J)=c_1(\eta)$. The pullback $J'=g^*J$ is another complex structure on $V$ and we have $c_1(J')=g^* c_1(J)=2 c_1(J)$ because (by Poincaré duality) $g$ (as well as $f$) acts by multiplication by $2$ on $\H^2(V) \simeq \H^2(\L(5,1)) \simeq \H_1(\L(5,1)) \simeq \pi_1 \L(5,1)$. Since $V$ has a Morse function with critical points of index $\leq 3$, theorem 13.1 of \cite{CE2012} allows us to construct a Weinstein structure on $V$ formally homotopic to $J'$; it induces a contact structure $\xi'$ on $M$. By pushing-forward by $g$, we get a Weinstein structure on $g(V) \subset V$. Since $W$ is an h-cobordism, as argued in \cite{C2014} the conditions of theorem 13.1 from \cite{CE2012} are met and we can construct a \emph{flexible} Weinstein structure on $W$ that extends that of $g(V)$. Hence we get a Weinstein structure on $V$ formally homotopic to $J$; it induces another contact structure $\xi$ on $M$. It then follows from a Mazur trick argument (see \cite{C2014}) that the symplectizations of $(M, \xi)$ and $(M, \xi')$ are exact symplectomorphic and also that the completions of $g(V)$ and $V$ are exact symplectomorphic.

\textit{Step 3 : Proof that the contact and Weinstein structures are not conjugate.}

We will show in fact that $c_1(\xi)$ and $c_1(\xi')$ are not conjugate by a diffeomorphism. Assume for contradiction that $\Psi : M \to M$ is a diffeomorphism such that $\Psi^*c_1(\xi)=c_1(\xi')$; by analyzing the action of $\Psi$ on cohomology we will show that $\Psi$ necessarily acts on $\pi_1$ by multiplication by $\pm 2$. Since $\H^*(\S^{2n-4})$ is free, we have a Künneth isomorphism (of graded rings)~:

\[\H^*(M) \overset{\sim}{\longrightarrow} \H^*(\L(5,1)) \otimes \H^*(\S^{2n-4}).\]

The inclusion $i: M \to V$ induces an isomorphism 
\[\H^2(V) \overset{\sim}{\longrightarrow} \H^2(\L(5,1))\otimes \H^0(\S^{2n-4})\simeq \Z/5;\] and we have $c_1(\xi)=i^*c_1(J)\neq 0$ and $c_1(\xi')=i^*c_1(J')=i^* (2 c_1(J))=2c_1(\xi)$. In degree $2n-4$, choose a generator $a$ of $\H^{0}(\L(5,1)) \otimes \H^{2n-4}(\S^{2n-4})\simeq \Z$, we have $\Psi(a) = \pm a + \alpha c_1(\xi)$ for $\alpha \in \Z/5$ if $n=3$ and $\Psi(a)=\pm a$ if $n>3$. Then $c_1(\xi) \cup a$ generates $\H^{2n-2}(M) \simeq \Z/5$ and we have:
\[\Psi^* (c_1(\xi) \cup a)=\Psi^*c_1(\xi) \cup \Psi^*a = c_1(\xi') \cup \Psi^*a = 2 c_1(\xi) \cup \Psi^*a=\pm 2 c_1(\xi) \cup a.\]
Hence, by Poincaré duality, $\Psi$ induces multiplication by $2$ on $\H_1(M) \simeq \H^{2n-2}(M)$, in contradiction with lemma \ref{simple} above.

Likewise if $\Psi : V \to V$ is a diffeomorphism that conjugates $J$ and $J'$, then $\Psi^*c_1(J')=c_1(J)$, and then $\Psi$ acts by multiplication by $2$ on $\H^2(V)\simeq \H_1(V)\simeq \pi_1(V)$, so cannot be homotopic to a diffeomorphism according to lemma \ref{simple}.
\end{proof}

\bibliographystyle{amsalpha}
\bibliography{/Users/sylvaincourte/Documents/Maths/Latex/biblio}

\providecommand{\bysame}{\leavevmode\hbox to3em{\hrulefill}\thinspace}
\providecommand{\MR}{\relax\ifhmode\unskip\space\fi MR }
\providecommand{\MRhref}[2]{%
  \href{http://www.ams.org/mathscinet-getitem?mr=#1}{#2}
}
\providecommand{\href}[2]{#2}
\begin{thebibliography}{dRMK67}

\bibitem[CE12]{CE2012}
Kai Cieliebak and Yakov Eliashberg, \emph{From {S}tein to {W}einstein and
  back}, American Mathematical Society Colloquium Publications, vol.~59,
  American Mathematical Society, Providence, RI, 2012.

\bibitem[Coh73]{Coh73}
Marshall~M. Cohen, \emph{A course in simple-homotopy theory}, Springer-Verlag,
  New York, 1973, Graduate Texts in Mathematics, Vol. 10.

\bibitem[Cou14]{C2014}
Sylvain Courte, \emph{Contact manifolds with symplectomorphic
  symplectizations}, Geometry \& Topology \textbf{18} (2014), no.~1, 1--15.

\bibitem[dRMK67]{dRMK67}
Georges de~Rham, Serge Maumary, and Michel~A. Kervaire, \emph{Torsion et type
  simple d'homotopie}. Lecture Notes in Mathematics, No. 48,
  Springer-Verlag, Berlin-New York, 1967.

\bibitem[Hae61]{Hae61}
Andr{{\'e}} Haefliger, \emph{Plongements diff{\'e}rentiables de
  vari{\'e}t{\'e}s dans vari{\'e}t{\'e}s}, Comment. Math. Helv. \textbf{36}
  (1961), 47--82.

\bibitem[HL76]{MR0415640}
Allen Hatcher and Terry Lawson, \emph{Stability theorems for ``concordance
  implies isotopy'' and ``{$h$}-cobordism implies diffeomorphism''}, Duke Math.
  J. \textbf{43} (1976), no.~3, 555--560.

\bibitem[Ker65]{K65}
Michel~A. Kervaire, \emph{Le th{\'e}or{\`e}me de Barden-Mazur-Stallings},
  Comment. Math. Helv. \textbf{40} (1965), 31--42.

\bibitem[Mil61]{Mil61}
John Milnor, \emph{Two complexes which are homeomorphic but combinatorially
  distinct}, Ann. of Math. (2) \textbf{74} (1961), 575--590.

\bibitem[Mur12]{M2012a}
Emmy Murphy, \emph{Loose legendrian embeddings in high dimensional contact
  manifolds}, arXiv:1201.2245 (2012).

\end{thebibliography}

\end{document}